\documentclass[12pt,reqno]{amsart}
\usepackage{graphicx}
\usepackage{verbatim}
\usepackage{textcomp}
\usepackage{amssymb}
\usepackage{cite}
\usepackage{amsmath}
\usepackage{latexsym}
\usepackage{amscd}
\usepackage{amsthm}
\usepackage{mathrsfs}
\usepackage{xypic}
\usepackage{bm}
\usepackage{url}
\usepackage{hyperref}
\usepackage{amsfonts}
\usepackage{times,graphicx,hyperref,mathrsfs}
\usepackage{color}
\newtheorem{thm}{Theorem}[section]

\newtheorem{lem}[thm]{Lemma}

\theoremstyle{definition}
\newtheorem{defn}{Definition}[section]

\theoremstyle{remark}
\newtheorem{rem}{\bf Remark}[section]
\numberwithin{equation}{section}
\allowdisplaybreaks
\begin{document}

\title[{ Fourier integral operators}]
{Notes on Regularity of Fourier integral operators with symbol in $S^{m}_{0,\delta}$ }

\author{Guangqing Wang$^{1,3}$}
\author{Suixin He$^{*2}$}
\address{$1.$ School of Mathematics and Statistics, Fuyang Normal University, Fuyang, Anhui 236041, P.R.China}
\address{$2.$ School of Mathematics and Statistics, Yili Normal University, Yili, Xinjiang 835000, P.R.China}
\address{$3.$School of Mathematics, Sun Yat-sen University, Guangzhou, Guangdong 510275, P.R.China}
\email{wanggqmath@fynu.edu.cn(G.Wang)}
\email{hesuixinmath@126.com(S.He)}


\thanks{*Corresponding  Author: Suixin He}
\thanks{The first author is supported by Scientific Research Foundation of Education Department of Anhui Province of China (2022AH051320), Doctoral Scientific Research Initiation Project of Fuyang Normal University (2021KYQD0001,2021KYQD0002), University-level Course Project of Fuyang Normal University (2025XTTZKC05), Research Project on Educational and Teaching Reform of Fuyang Normal University (2024JYXM0018) and the second author is supported by Natural Science Foundation of Xinjiang Uygur Autonomous Region(2022D01C734).}
\maketitle

\begin{abstract}
Let $T_{a,\varphi}$ be a Fourier integral operator defined with $a\in S^{m}_{0,\delta}(0\leq\delta<1)$ and $\varphi\in \Phi^{2}$ satisfying the strong non-degenerate condition. We demonstrate that when the order satisfies
$$m\leq-\frac{n}{2}-\frac{n}{p}\delta+\frac{n}{p},$$
the operator $T_{a,\varphi}$ becomes bounded on $L^{p}(\mathbb{R}^n)$ for $2< p<\infty$ and maps $L^{\infty}(\mathbb{R}^n)$ to $BMO(\mathbb{R}^n)$ when $p=\infty$. Furthermore, the derived bound on $m$ is sharp for $L^{p}$ estimates in the case $\delta=0$, and for $(L^{\infty},BMO)$ when $0\leq\delta<1$.
\end{abstract}



\section{Introduction and main results}

Fourier integral operators (FIOs), first introduced by H\"{o}rmander \cite{HormanderI}, are formally defined as
\begin{eqnarray}\label{df}
T_{a,\varphi}u(x)
&=&\frac{1}{(2\pi)^n}\int_{\mathbb{R}^n} e^{ i \varphi(x,\xi)}a(x,\xi)  \hat{u}(\xi)d\xi,
\end{eqnarray}
where $\varphi\in C^{\infty}(\mathbb{R}^{n}\times\mathbb{R}^n\setminus\{0\})$ is a phase function that is homogeneous of degree 1 in $\xi$ and satisfies some non-degeneracy condition, and $a(x,\xi)$ belongs to the H\"{o}rmander class $S^{m}_{\varrho,\delta}$ with parameters $m\in \mathbb{R}$, $0\leq\varrho,\delta\leq1$. Specifically, $a\in S^{m}_{\varrho,\delta}$ if it satisfies the symbol estimates
\begin{eqnarray*}
|\partial^{\beta}_x\partial^{\alpha}_{\xi}a(x,\xi)|\leq C_{\alpha,\beta}(1+|\xi|)^{m-\varrho |\alpha|+\delta|\beta|},
\end{eqnarray*}
for any multi-indices $\alpha,\beta.$

A central problem in FIOs theory concerns their regularity properties in various function spaces, which has been studied extensively in \cite{Eskin,HormanderI,Hormander2,Tao,Beals,Littman,Miyachi,Beals2,Peral,Greenleaf,Wolfgang,SSS,WCY,SZ,Wolfgang1,Israelsson}. Fundamental $L^{2}$-estimates for FIOs were established by \`{E}skin\cite{Eskin} and H\"{o}rmander\cite{HormanderI} under the condition $\frac{1}{2}<\delta=\varrho\leq1,$ later extended to the full range $0\leq\varrho\leq1$ and $0\leq\delta<1$ by Ferreira et al. \cite{Wolfgang}, achieving sharpness in $m\leq -\min(0,\frac{n}{2}(\varrho-\delta))$. $L^{p}$-regularity was first obtained by Seeger et al. \cite{SSS} for $\varrho=1$, $\delta=0$, and subsequently generalized to $0<\varrho\leq1$ and $0\leq\delta<1$ in \cite{WCY,SZ,Wolfgang1}. Recent work by Israelsson et al. \cite{Israelsson} established $h^{p}\rightarrow L^{p}$ boundedness for FIOs with $\varrho=0$, which forms our starting point:

\begin{thm}{\rm \cite[Proposition 6.2]{Israelsson}}\label{D0}
Let \( T_{a,\varphi}\) be an FIO with amplitude \( a \in S_{0, \delta}^{m} (\mathbb{R}^n) \) for $0\leq\delta<1$ and
\[
m\leq -n \left| \frac{1}{p} - \frac{1}{2} \right| - \frac{n}{2}\delta,
\] and let \( \varphi \in \Phi^2 \) be a strongly non-degenerate (SND) phase function of rank \( \kappa \in[0,n-1]\). Then \( T_{a,\varphi}: h^p(\mathbb{R}^n) \to L^p(\mathbb{R}^n) \) for \( 0 < p < \infty \) and \( T_{a,\varphi}: L^\infty(\mathbb{R}^n) \to \mathrm{bmo}(\mathbb{R}^n) \) when \( p = \infty \). Here  $h^{p}(\mathbb{R}^n)$ and $bmo(\mathbb{R}^n)$ denote the local Hardy space and local BMO space respectively.
\end{thm}

\renewcommand{\thethm}{\thesection.\arabic{thm}}
In this paper our main contribution is to improve the bound on $m$ in the case when $2< p\leq\infty$.
\begin{thm}\label{T1}
 Suppose $2< p\leq\infty, 0\leq\delta<1$ and
$$m\leq-\frac{n}{2}-\frac{n}{p}\delta+\frac{n}{p}.$$
Let \( T_{a,\varphi}\) be an FIO with amplitude function \( a \in S_{0, \delta}^{m} (\mathbb{R}^n) \) and phase function \( \varphi \in \Phi^2 \) satisfying SND condition. Then \( T_{a,\varphi}: L^p(\mathbb{R}^n) \to L^p(\mathbb{R}^n) \) for \( 2 < p < \infty \) and \( T_{a,\varphi}: L^\infty(\mathbb{R}^n) \to \mathrm{BMO}(\mathbb{R}^n) \) when \( p = \infty \).
\end{thm}

\begin{rem}
The bound on $m$ is sharp: For the $L^{p}(2 < p < \infty)$ estimate when $\delta=0$, and for $(L^{\infty},BMO)$ when $0\leq\delta<1$. This coincides with optimal results for pseudo-differential operators (PDOs), see \cite{W,Miyachi1}.
For the case when $0<\varrho\leq1$, the bound on $m$ is probably optimal\cite{SZ,WCY,Yang}.
\end{rem}

\begin{lem}
By adapting our proof strategy for Theorem \ref{T1} and the techniques from \cite[Theorem 1.1]{W}, one can establish $(H^{1},L^{1})$ estimates for FIOs and PDOs with the bound $m\leq -\frac{n}{2}-\frac{n}{2}\delta$ respectively. When $\delta=0$, this bound is sharp also\cite{W,Miyachi1}.
\end{lem}

The $L^{p}$-regularity of these FIOs follows from $(L^{\infty},{\rm BMO})$ estimates via complex interpolation\cite{Fefferman,Stein}. One has to point out that FIOs lack off-diagonal kernel regularity in the limiting case $\varrho=0$. Existing approaches to $(L^{\infty},{\rm BMO})$ estimates rely on duality arguments and interpolation between the $L^{2}$-estimate and the $(H^{p},L^{p})$ estimate for some $0<p<1$ \cite{Miyachi1,Wolfgang1}. We emphasize that weak-type estimates for FIOs and PDOs for $\varrho=0$ remain open, see \cite{Hounie1,Tao}. Our new proof, inspired by pseudo-differential operators(PDOs) techniques in \cite{W,W1}, circumvents these limitations. Our approach may enable us to study the regularity of Fourier integral operators in other spaces. For instance, weighted Lebesgue spaces, mixed-norm Lebesgue spaces and Morrey spaces. We recommend \cite{Haroske,Huang,2Yang} for the latest related developments regarding these spaces.

The paper is organized as follows: Section 2 provides necessary preliminaries, while Section 3 contains detailed proofs.
\section{Notations and Preliminaries}

Given a Lebesgue measurable set \( E \), the Lebesgue measure of \( E \) is denoted by \( |E| \), and the characteristic function of \( E \) is denoted by \( \kappa_{E}(x) \). Positive constants in this article are represented by the letter \( C \), with their values allowed to vary from line to line. For convenience, we use the notation \( A \lesssim B \) as shorthand for \( A \leq CB \).

In this investigation, we deal with following phase functions whose definition is attributed to Ferreira and Staubach \cite{Wolfgang}.

\begin{defn}
\textit{\((\Phi^k)\)}
A real-valued function \( \varphi(x,\xi) \) belongs to the class \( \Phi^k \) if the following conditions hold:
\begin{itemize}
    \item \( \varphi(x,\xi) \in C^\infty(\mathbb{R}^n \times \mathbb{R}^n \setminus \{0\}) \),
    \item \( \varphi(x,\xi) \) is positively homogeneous of degree 1 in the frequency variable \( \xi \),
    \item For any pair of multi-indices \( \alpha \) and \( \beta \) with \( |\alpha| + |\beta| \geq k \), there exists a positive constant \( C_{\alpha,\beta} \) such that:
\[
\sup_{(x,\xi) \in \mathbb{R}^n \times \mathbb{R}^n \setminus \{0\}} |\xi|^{-1 + |\alpha|} |\partial_{\xi}^\alpha \partial_{x}^\beta \varphi(x,\xi)| \leq C_{\alpha,\beta}.
\]
\end{itemize}
\end{defn}

For establishing global regularity results, the following strong non-degeneracy condition is required:

\begin{defn}
\textit{(SND condition)}
A real-valued function \( \varphi(x,\xi) \in C^2(\mathbb{R}^n \times \mathbb{R}^n \setminus \{0\}) \) satisfies the \textit{strong non-degeneracy condition} if there exists a positive constant \( c \) such that:
\[
\left| \det \frac{\partial^2 \varphi(x,\xi)}{\partial x_j \partial \xi_k} \right| \geq c,
\]
for all \( (x,\xi) \in \mathbb{R}^n \times \mathbb{R}^n \setminus \{0\} \).
\end{defn}
Next, we make some preparations for the proof of our main result. The fundamental \( L^{2} \) estimate for FIOs is necessary and will be frequently used in the following statement.

\begin{lem}\cite[Theorem 2.7]{Wolfgang}\label{P1}
Suppose that \( 0 \leq \varrho \leq 1 \), \( 0 \leq \delta < 1 \), and \( m = \min\big(0, \frac{n}{2}(\varrho - \delta)\big) \). Assume that \( a \in S^{m}_{\varrho, \delta} \) and \( \varphi \in \Phi^{2} \) satisfies the strongly non-degenerate condition. Then the FIO defined as in (\ref{df}) extends as a bounded operator on \( L^{2} \).
\end{lem}

We start with a dyadic partition of unity. Let \( A = \{\xi \in \mathbb{R}^n : \frac{1}{2} \leq |\xi| \leq 2\} \) be an annulus, \( \chi_0(\xi) \in C^\infty_0(B(0,2)) \), and define \( \chi_j(\xi) = \chi(2^{-j}\xi) \) for \( j \geq 1 \), where \( \chi(\xi) \in C^\infty_0(A) \). Then we have
\begin{equation}\label{E0}
    \chi_0(\xi) + \sum_{j=1}^{\infty} \chi_j(\xi) = 1 \quad \text{for all} \quad \xi \in \mathbb{R}^n.
\end{equation}

Using this partition of unity, the Fourier integral operator can be decomposed as:
\begin{equation}\label{E1}
    T_{a,\varphi} = T_{0} + \sum_{j=1}^{\infty} T_{j}.
\end{equation}

The kernel of the operator \( T_{j} \) is then given by:
\begin{equation*}
    T_j(x, y) = \int_{\mathbb{R}^n} e^{i(\varphi(x, \xi) - y \cdot \xi)} \chi_j(\xi) a(x, \xi) \, d\xi.
\end{equation*}

Let $\{\xi^\nu_{0}\}_{1\leq\nu\leq J}$ denote collection of unit vectors satisfying
\begin{align*}
  \textrm{(i)} & \quad\quad \left|\xi^\nu_{0} -\xi^\mu_{0}\right|\geq 1; \\
 \textrm{(ii)} & \quad\quad \text{If}~~ \xi\in \mathbb{S}^{n-1},  \text{then there exists a}\;\xi^\nu_{0}\; \text{so that } |\xi-\xi^\nu_{0}|\leq 1,
\end{align*}

Let $\Gamma^\nu_{0}=\{\xi\in\mathbb{R}^n:|\frac{\xi}{|\xi|}-\xi^{\nu}_{0}|\leq1 \}$,
one can construct an associated partition of unity given by functions $\psi^\nu$,  each homogeneous of degree 0 in $\xi$ with $\sum_{\nu=1}^J\psi^\nu(\xi)=1$ for all $\xi\neq0$,  and satisfy $\text{supp} \psi^\nu\subset \Gamma^{\nu}_{0}.$ It is clearly that $J\sim1.$ Write
\begin{equation}\label{Tj}
T_{j}=\sum_{\nu}T^{\nu}_{j},
\end{equation}
where  the operator $ T_j^\nu$ is given by
 \begin{equation}\label{Tjv}
    T_j^\nu u(x)= \int_{{\mathbb{R}^n}}T_j^\nu(x, y)u(y)dy.
 \end{equation}
with its kernel
 \begin{equation}\label{Tjvker0}
    T_j^\nu(x, y)= \int_{{\mathbb{R}^n}}e^{ i( \varphi(x, \xi)-y\cdot\xi)}\chi_j(\xi)a(x,\xi)\psi^\nu(\xi)d\xi.
 \end{equation}
Write
\begin{eqnarray*}
   \varphi(x, \xi)-y\cdot\xi&=&  \langle\nabla_\xi\varphi(x, \xi^{\nu}_{0})-y, \xi\rangle+\varphi(x, \xi)-\langle\nabla_\xi\varphi(x, \xi^{\nu}_{0}), \xi\rangle.
\end{eqnarray*}
Denote $
a_j^\nu(x, \xi)
=\chi_j(\xi)a(x,\xi)\psi^\nu(\xi)e^{ i\cdot2^{j\varrho }[\varphi(x, \xi)-\langle\nabla_\xi\varphi(x, \xi^{\nu}_{0}), \xi\rangle]}.
$
Then the kernel defined in (\ref{Tjvker0}) can be written as
 \begin{equation}\label{Tjvker}
    T_j^\nu(x, y)= \int_{{\mathbb{R}^n}}e^{ i\langle\nabla_\xi\varphi(x, \xi^{\nu}_{0})-y, \xi\rangle}a^{\nu}_j(x, \xi)d\xi.
 \end{equation}
Clearly, $a_j^\nu(x, \xi)$'s support in $\xi$ is contained in
 $$A_{j}\cap\Gamma^{\nu}_{0}=\left\{\xi: 2^{-1}2^{j}<|\xi|<2\cdot2^{j}  \right\}\cap\{\xi\in\mathbb{R}^n:|\frac{\xi}{|\xi|}-\xi^{\nu}_{0}|\leq1 \},$$
 where $A_j=\left\{\xi: 2^{-1}2^{j}<|\xi|<2\cdot2^{j}  \right\}.$ Choosing a suitable coordinates with respect to $\xi$, one can get the following estimate immediately
\begin{equation*}
 \sup_{\xi\in {\mathbb{R}^n}}|\partial^\alpha\psi^\nu(\xi)|\leq C_\alpha2^{-j}~{\rm and}~
\left|\partial^{\beta }_\xi (\varphi(x, \xi)-\langle\nabla_\xi\varphi(x, \xi^{\nu}_{0}), \xi\rangle)\right|\leq C_\beta.
\end{equation*}
for any $|\alpha|\geq0$ and $\beta\geq1$. Leibniz's formula gives the following fundamental estimate for $a_j^\nu(x, \xi)$.
\begin{lem}\label{L0}
Assume $a\in S^{-\frac{n}{2}}_{0,\delta}$ with $0\leq\delta<1 $ and that $\varphi\in \Phi^{2}$ satisfies the strongly non-degenerate condition.  Then, the following estimates hold
\begin{equation}\label{Ec}
 \sup_{\xi}\left|\partial^\alpha_\xi a_j^\nu (x, \xi)\right|\leq C_\alpha2^{-j\frac{n}{2}}.
\end{equation}
\end{lem}
This estimate is a variant of similar result in \cite[Ch.9 \S 4.5.]{Stein}. See also \cite{SSS} or \cite[Lemma 1.9]{Wolfgang}.

Let $Q=Q(x_{0},l)$ be the cube about $x_{0}$ with radius $l>0$. Now, we give some local estimates with respect to $Q$ for the operators $T_{j}$.
\begin{lem}\label{L1}
Assume $a\in S^{-\frac{n}{2}}_{0,\delta}$ with $0\leq\delta<1 $ and that $\varphi\in \Phi^{2}$ satisfies the strongly non-degenerate condition.  Then, the following estimates hold
\begin{eqnarray}
&&\int_{\mathbb{R}^{n}}|T_{j}(x,y)-T_{j}(z,y)|dy\lesssim 2^{j}l, \quad {\rm all}~ x,z\in Q;\label{E6}
\end{eqnarray}
\end{lem}
\begin{proof}
By the decomposition (\ref{Tj}) and the fact that $J\sim1$, it is sufficient to show that $T^{\nu}_{j}(x,y)$ meets the inequality (\ref{E6}). Direct computation gives
\begin{eqnarray*}
&&\int_{\mathbb{R}^{n}}|T^{\nu}_{j}(x,y)-T^{\nu}_{j}(z,y)|dy\\
&\leq&l\int_{\mathbb{R}^{n}}|\int_{\mathbb{R}^n}e^{ -i\langle y, \xi\rangle}\nabla_{x}a^{\nu}_{j}(\bar{x},\xi)d\xi|dy\\
&&+l\int_{\mathbb{R}^{n}}|\int_{\mathbb{R}^n}e^{-i\langle y, \xi\rangle}\langle\nabla_x\nabla_\xi\varphi(x', \xi^{\nu}_{0}), \xi\rangle a^{\nu}_{j}(x,\xi)d\xi|dy=:I_{1}+I_{2},
\end{eqnarray*}
where $\bar{x}$ and $x'$ denote some points between $x$ and $z$.
By $\varphi\in \Phi^{2}$, $a\in S^{-\frac{n}{2}}_{0,\delta}$ and (\ref{Ec}), It is clear that
\begin{equation}\label{E7}
  \|\partial^\alpha_\xi \nabla_{x}a^{\nu}_{j}(\bar{x},\xi)\|_{L^{\infty}}\lesssim 2^{j(-\frac{n}{2}+1)}~{\rm and}~
  \|\partial^\alpha_\xi \langle\nabla_x\nabla_\xi\varphi(x', \xi^{\nu}_{0}), \xi\rangle a^{\nu}_{j}(x,\xi)\|_{L^{\infty}}\lesssim 2^{j(-\frac{n}{2}+1)}.
\end{equation}

Next we proceed with the estimate of $I_{1}$ only and $I_{2}$ can be estimated by a similar argument. To this end, we write by Integrating by parts
\begin{eqnarray*}
l\int_{\mathbb{R}^{n}}|\int_{{\mathbb{R}^n}}
e^{ -i\langle y, \xi\rangle}\nabla_{x}a^{\nu}_{j}(\bar{x},\xi)d\xi| dy
&=&l\int_{\mathbb{R}^{n}}(1+|y|^{2})^{-N}|\widehat{(1-\partial^{2}_{\xi})^{N}\nabla_{x}a^{\nu}_{j}(\bar{x},\cdot)}(y)|dy.
\end{eqnarray*}
for any positive real number $N>0.$ The Cauchy-Schwarz inequality and the Plancherel theorem give that it is bounded by
\begin{eqnarray*}\label{Ed}
l\bigg(\int_{\mathbb{R}^{n}}|(1-\partial^{2}_{\xi})^{N}\nabla_{x}a^{\nu}_{j}(\bar{x},\xi)|^{2}d\xi\bigg)^{\frac{1}{2}}
&\lesssim&
l2^{j(-\frac{n}{2}+1)}|A_{j}\cap\Gamma^{\nu}_{0}|^{\frac{1}{2}}\lesssim l2^{j}.
\end{eqnarray*}
Here we chose $N>\frac{n}{4}$, and the fact (\ref{E7}) and $|A_{j}\cap\Gamma^{\nu}_{0}|\lesssim 2^{jn}$ are used. So the proof is finished.
\end{proof}

\begin{lem}\label{L2}
Suppose $0<l<1$, $a\in S^{-\frac{n}{2}}_{0,0}$ and that $\varphi\in \Phi^{2}$ satisfies the strongly non-degenerate condition. For any positive integer $N>\frac{n}{2}$ and any positive integer $j$ with $l^{-1}\leq2^{j}$, then the following estimates hold
\begin{eqnarray}
\frac{1}{|Q(x_{0},l)|}\int_{Q(x_{0},l)}|T_{j}u(x)|dx
\lesssim2^{-j\frac{n}{2}(1-\frac{n}{2N})}l^{\frac{n}{2}(\frac{n}{2N}-1)}\|u\|_{L^{\infty}}.\label{E131}
\end{eqnarray}
\end{lem}

\begin{proof}
By the decomposition (\ref{Tj}) and the fact that $J\sim1$ again, we only show that $T^{\nu}_{j}$ meets the inequality (\ref{E131}).
To this end, fix positive integral $N>\frac{n}{2}$ and
set
\begin{equation}\label{R}
R^{\nu}_{j}=\{y:|\nabla_{\xi}\varphi(x_{0},\xi^{\nu}_{0})-y|\leq\bar{c}l^{\frac{n}{2N}}2^{\frac{jn}{2N}}\},
\end{equation}
where $\bar{c}$ is a large constant independent of $j,l$ and it will be fixed later. Set
\begin{eqnarray*}
u^{\nu}_{j,1}(x)=u(x)\kappa_{R^{\nu}_{j}}(x) \quad{\rm and} \quad u^{\nu}_{j,2}(x)=u(x)-u^{\nu}_{i,1}(x),
\end{eqnarray*}
where $\kappa_{R^{\nu}_{j}}(x)$ is the characteristic function of the rectangle $R^{\nu}_{j}.$
Then one can write
$$T^{\nu}_{j}u(x)=T^{\nu}_{j}u^{\nu}_{j,1}(x)+T^{\nu}_{j}u^{\nu}_{j,2}(x).$$

It is easy to check that $\chi_j(\xi)a(x,\xi)\psi^\nu(\xi)\in S^{0}_{0,0}$ with the bounds $\lesssim2^{-j\frac{n}{2}}$. H\"{o}lder's inequality and $L^{2}$-boundedness of $T^{\nu}_{j}$ imply that
\begin{eqnarray}\label{b11}
\frac{1}{|Q(x_{0},l)|}\int_{Q(x_{0},l)}|T^{\nu}_{j}u^{\nu}_{j,1}(x)|dx\leq l^{-\frac{n}{2}}\|T^{\nu}_{j}u^{\nu}_{j,1}\|_{L^{2}}\nonumber
&\lesssim& 2^{-j\frac{n}{2}}l^{-\frac{n}{2}}\|u^{\nu}_{j,1}\|_{L^{2}}\\
&\lesssim&2^{j\frac{n}{2}(\frac{n}{2N}-1)}l^{\frac{n}{2}(\frac{n}{2N}-1)}\|u\|_{L^{\infty}},
\end{eqnarray}
where the inequality $|R^{\nu}_{j}|\lesssim l^{n}$ is used.

Next, we prove that $T^{\nu}_{j}u^{\nu}_{j,2}(x)$ meets the inequality $\ref{E131}$ as well. Note that if $x\in Q,$ then by the fact that $\varphi\in \Phi^{2}$ and $l\leq1$
$$|\nabla_\xi\varphi(x,\xi^{\nu}_{0})-\nabla_\xi\varphi(x_{0},\xi^{\nu}_{0})|\leq C_{1,1}|x-x_{0}|\leq C_{1,1}l\leq C_{1,1}.$$
For any $y\in~^{c}R^{\nu}_{j}$, one can get further
 \begin{equation*}
|\nabla_\xi\varphi(x_{0},\xi^{\nu}_{0})-y|>\bar{c}l^{\frac{n}{2N}}2^{\frac{jn}{2N}}\geq\frac{\bar{c}}{C_{1,1}}|\nabla_\xi\varphi(x,\xi^{\nu}_{0})-\nabla_\xi\varphi(x_{0},\xi^{\nu}_{0})|.
\end{equation*}
Here, the fact $2^{j}\geq l^{-1}$ is used.
Taking $\bar{c}$ large enough, we therefore have
\begin{eqnarray*}
|\nabla_\xi\varphi(x,\xi^{\nu}_{0})-y|
&\geq &|\nabla_\xi\varphi(x_{0},\xi^{\nu}_{0})-y|-|\nabla_\xi\varphi(x_{0},\xi^{\nu}_{0})-\nabla_\xi\varphi(x,\xi^{\nu}_{0})|\\
&\gtrsim & |\nabla_\xi\varphi(x_{0},\xi^{\nu}_{0})-y|.
\end{eqnarray*}
Inserting this inequality into
$$T_j^\nu u^{\nu}_{j,2}(x)= \int_{^{c}R^{\nu}_{j}}T_j^\nu(x, y)u(y)dy,$$
one can get that for any $x\in Q,$
\begin{eqnarray*}
|T^{\nu}_{j}u^{\nu}_{j,2}(x)|
&\leq& \|u\|_{L^{\infty}}
\int_{^{c}R^{\nu}_{j}}\frac{|\nabla_\xi\varphi(x,\xi^{\nu}_{0})-y|^{N}}{|\nabla_\xi\varphi(x_{0},\xi^{\nu}_{0})-y|^{N}}|T^{\nu}_{j}(x,y)| dy.
\end{eqnarray*}
H\"{o}lder's inequality, Integrating by parts  and Parseval's identity give that
\begin{eqnarray}\label{b12}
|T^{\nu}_{j}u^{\nu}_{j,2}(x)|
&\leq& \|u\|_{L^{\infty}}\big(\int_{^{c}R^{\nu}_{j}}\frac{1}{|\nabla_\xi\varphi(x_{0},\xi^{\nu}_{0})-y|^{2N}} dy\big)^{\frac{1}{2}}
\big(\int_{\mathbb{R}^{n}}\big|\widehat{\partial^{N}_{\xi}a_j^\nu(x, \cdot)}(y
)\big|^{2} dy\big)^{\frac{1}{2}}\nonumber\\
&\lesssim& \|u\|_{L^{\infty}}\big(l^{\frac{n}{2N}}2^{j\frac{n}{2N}}\big)^{\frac{n}{2}-N}
2^{-j\frac{n}{2}}
|A_{j}\cap\Gamma^{\nu}_{0}|^{\frac{1}{2}}
\lesssim2^{j\frac{n}{2}(\frac{n}{2N}-1)}l^{\frac{n}{2}(\frac{n}{2N}-1)}\|u\|_{L^{\infty}},
\end{eqnarray}
where the fact that $|A_{j}\cap\Gamma^{\nu}_{0}|\lesssim 2^{jn}$ and (\ref{Ec}) are used. (\ref{b11}) and (\ref{b12}) imply that $T^{\nu}_{j}$ meets the inequality (\ref{E131}). So the proof is finished.
\end{proof}

\begin{lem}\label{L3}
Suppose that $0<l<1$, $a\in S^{-\frac{n}{2}}_{0,\delta}$ with $0<\delta<1 $ and that $\varphi\in \Phi^{2}$ satisfies the strongly non-degenerate condition.  Then for any $\lambda \geq1$, any positive integer $N>\frac{n}{2}$ and any positive integer $j$ with $l^{-\lambda}\leq2^{j}$, we have
\begin{eqnarray*}
\frac{1}{|Q(x_{0},l)|}\int_{Q(x_{0},l)}|T_{j}u(x)|dx
&\lesssim&\|u\|_{L^{\infty}}\big(l^{\lambda}2^{j\delta}+2^{j\frac{n}{2}(\frac{n}{2N}-2)}l^{\lambda\frac{n}{2}(\frac{n}{2N}-2)}\big).
\end{eqnarray*}
\end{lem}
\begin{proof}
Assume $\lambda>1$. Then $l^{\lambda}<l$ since $l<1$. Define $L$ as the smallest integer greater than or equal to $l^{1-\lambda}$, i.e., $L-1<l^{1-\lambda}\leq L$. Using this,$Q(x_{0},l)$ can be covered by $L^{n}$ cubes with side length $l^{\lambda}$, and we have
$$Q(x_{0},l)\subset\cup_{i=1}^{L^{n}}Q(x_{i},l^{\lambda})\subset Q(x_{0},2l)$$
Clearly, the number of such cubes $L^{n}$ less than $2^{n}l^{n(1-\lambda)}$

Define
\begin{eqnarray}\label{0A}
T_{j,i}u(x)
&=&\int_{\mathbb{R}^n} e^{ i \phi(x,\xi)}a(x_{i},\xi) \chi_{j}(\xi) \hat{u}(\xi)d\xi.
\end{eqnarray}
We write
\begin{eqnarray}\label{A}
&&\frac{1}{|Q|}\int_{Q(x_{0},r)}|T_{j}u(x)|dx\nonumber\\
&\leq&\frac{1}{|Q|}\sum\limits_{i=1}^{L^{n}}\bigg(\int_{Q(x_{i},r^{\lambda})}|T_{j}u(x)-T_{j,i}u(x)|dx+\int_{Q(x_{i},r^{\lambda})}|T_{j,i}u(x)|dx\bigg).
\end{eqnarray}
We claim that
\begin{eqnarray}\label{a}
|T_{j}u( x)-T_{j,i}u( x)|\lesssim \|u\|_{L^{\infty}}|x-x_{i}|2^{j\delta}
\end{eqnarray}
and
\begin{eqnarray}\label{b}
\int_{Q(x_{i},r^{\lambda})}|T_{j,i}u(x)|dx
&\lesssim&\|u\|_{L^{\infty}}2^{j\frac{n}{2}(\frac{n}{2N}-2)}l^{\lambda\frac{n}{2}(\frac{n}{2N}-2)}l^{\lambda n}.
\end{eqnarray}
Since $L^{n}\leq2^{n}l^{n(1-\lambda)}$, the desired estimate follows by substituting (\ref{a}) and (\ref{b}) into (\ref{A}). Notice that $l^{\lambda}<1$ and $a(x_{i},\xi)\in S^{-\frac{n}{2}}_{0,\delta}$ with its bound independent of $x_{i}$. Therefore, (\ref{b}) follows from Lemma \ref{L2}.

Now we prove (\ref{a}). For any fixed $x_{i}$, define
\begin{eqnarray*}
\tilde{a}(x,\xi)=a(x,\xi)-a(x_{i},\xi) \quad{\rm and} \quad \tilde{T_j}(x, y)= \int_{{\mathbb{R}^n}}e^{ i( \varphi(x, \xi)-y\cdot\xi)}\tilde{a}(x,\xi)\chi_j(\xi)d\xi.
\end{eqnarray*}
Then
\begin{eqnarray}\label{c}
T_{j}u(x)-T_{j,i}u(x)=\int_{\mathbb{R}^n} \tilde{T_j}(x, y) u(y)dy.
\end{eqnarray}
and
\begin{eqnarray*}
|\partial^{\alpha}_{\xi}\tilde{a}(x,\xi)|\leq C_{\alpha}\langle\xi\rangle^{-\frac{n}{2}+\delta}|x-x_{i}|,
\end{eqnarray*}
for any multi-indices $\alpha.$ Using similar methods, it is straightforward to show
\begin{eqnarray}\label{d}
\int_{\mathbb{R}^n}|\tilde{T_j}(x, y)|dy\lesssim2^{j\delta}|x-x_{i}|.
\end{eqnarray}
Clearly, (\ref{c}) and (\ref{d}) imply (\ref{a}).
If $\lambda=1$, define
\begin{eqnarray*}
T_{j,0}u(x)
&=&\int_{\mathbb{R}^n} e^{ i \phi(x,\xi)}a(x_{0},\xi) \chi_{j}(\xi) \hat{u}(\xi)d\xi..
\end{eqnarray*}
Then the desired estimate can be got by the same argument replacing$T_{j,i}u$ with $T_{j,0}u$. This completes the proof.
\end{proof}

By restricting $\lambda$ to $[1,\frac{1}{1-\delta}]$, the sum in (\ref{E2}) becomes convergent for $l^{-1}\leq2^{j}\leq l^{-\frac{1}{1-\delta}}$. For the case $2^{j}>l^{-\frac{1}{1-\delta}}$, the following lemma is essential.
\begin{lem}\label{L4}
Suppose that $0<l<1$, $a\in S^{-\frac{n}{2}}_{0,\delta}$ with $0<\delta<1 $ and that $\varphi\in \Phi^{2}$ satisfies the strongly non-degenerate condition.  Then any positive integer $N>\frac{n}{2}$ and any positive integer $j$ with $2^{j}>l^{-\frac{1}{1-\delta}}$, we have
\begin{eqnarray*}
\frac{1}{|Q(x_{0},l)|}\int_{Q(x_{0},l)}|T_{j}u(x)|dx
&\lesssim&\|u\|_{L^{\infty}}2^{-j\frac{n}{2}(1-\delta)(1-\frac{n}{2N})}l^{-\frac{n}{2}(1-\frac{n}{2N})}.
\end{eqnarray*}
\end{lem}
The proof of Lemma (\ref{L4}) is parallel to Lemma (\ref{L3}) after changing the definition of $R^{\nu}_{j}$ to
\begin{equation*}
R^{\nu}_{j}=\{y:|\nabla_{\xi}\varphi(x_{0},\xi^{\nu}_{0})-y|\leq\bar{c}2^{j\frac{n}{2N}(1-\delta)}l^{\frac{n}{2N}}\}.
\end{equation*}
And, the fact is used that $\chi_j(\xi)a(x,\xi)\psi^\nu(\xi)\in S^{-\frac{n}{2}\delta}_{0,\delta}$ with the bounds $\lesssim2^{-j\frac{n}{2}(1-\delta)}$.

For the case $l\geq1$, we have the following lemma.
\begin{lem}\label{L5}
Suppose that $l\geq1$, $a\in S^{-\frac{n}{2}}_{0,\delta}$ with $0\leq\delta<1 $ and that $\varphi\in \Phi^{2}$ satisfies the strongly non-degenerate condition.  Then any positive integer $N>\frac{n}{2}$ and $0<\theta<1-\delta$, we have
\begin{eqnarray}\label{E3}
\frac{1}{|Q(x_{0},l)|}\int_{Q(x_{0},l)}|T_{j}u(x)|dx
&\lesssim&\|u\|_{L^{\infty}}(2^{-j\frac{n}{2}(1-\delta-\theta)}+2^{-j\theta(N-\frac{n}{2})}l^{-(N-\frac{n}{2})}).
\end{eqnarray}
\end{lem}

\begin{proof}
By the decomposition (\ref{Tj}) and the fact that $J\sim1$, it is sufficient to show that $T^{\nu}_{j}$ meets the inequality (\ref{E3}).
To this end, set
\begin{equation}\label{R3}
\Bar{R}^{\nu}_{j}=\{y:|\nabla_{\xi}\varphi(x_{0},\xi^{\nu}_{0})-y|\leq\bar{c}l2^{j\theta}\}.
\end{equation}
\begin{eqnarray*}
u^{\nu}_{j,3}(x)=u(x)\chi_{\Bar{R}^{\nu}_{j}}(x) \quad{\rm and} \quad u^{\nu}_{j,4}(x)=u(x)-u^{\nu}_{i,3}(x),
\end{eqnarray*}
where $\chi_{\Bar{R}^{\nu}_{j}}(x)$ is the characteristic function of the rectangle $\Bar{R}^{\nu}_{j}.$
So one can write
$$T^{\nu}_{j}u(x)=T^{\nu}_{j}u^{\nu}_{j,3}(x)+T^{\nu}_{j}u^{\nu}_{j,4}(x).$$

Notice that $\chi_j(\xi)a(x,\xi)\psi^\nu(\xi)\in S^{-\frac{n}{2}\delta}_{0,\delta}$ with bounds $\lesssim 2^{-j\frac{n}{2}(1-\delta)}$. H\"{o}lder's inequality and the $L^{2}$-estimate of $T_{j}$ give that
\begin{eqnarray}\label{E4}
\frac{1}{|Q|}\int_{Q}|T^{\nu}_{j}u^{\nu}_{j,3}(x)|dx\lesssim 2^{-j\frac{n}{2}(1-\delta)}l^{-\frac{n}{2}}\|u^{\nu}_{j,3}\|_{L^2}\lesssim 2^{-j\frac{n}{2}(1-\delta-\theta)}\|u\|_{L^{\infty}}.
\end{eqnarray}

Note that if $x\in Q,$ then by the fact that $\varphi\in \Phi^{2}$
$$|\nabla_\xi\varphi(x,\xi^{\nu}_{0})-\nabla_\xi\varphi(x_{0},\xi^{\nu}_{0})|\leq C_{1,1}|x-x_{0}|\leq C_{1,1}l.$$
For any $y\in~^{c}\Bar{R}^{\nu}_{j}$, one can get further
 \begin{equation*}
|\nabla_\xi\varphi(x_{0},\xi^{\nu}_{0})-y|>\bar{c}l2^{j\theta}>\bar{c}l\geq\frac{\bar{c}}{C_{1,1}}|\nabla_\xi\varphi(x,\xi^{\nu}_{0})-\nabla_\xi\varphi(x_{0},\xi^{\nu}_{0})|.
\end{equation*}
Taking $\bar{c}$ large enough, we therefore have
\begin{eqnarray*}
|\nabla_\xi\varphi(x,\xi^{\nu}_{0})-y|
&\geq &|\nabla_\xi\varphi(x_{0},\xi^{\nu}_{0})-y|-|\nabla_\xi\varphi(x_{0},\xi^{\nu}_{0})-\nabla_\xi\varphi(x,\xi^{\nu}_{0})|\\
&\gtrsim & |\nabla_\xi\varphi(x_{0},\xi^{\nu}_{0})-y|.
\end{eqnarray*}
Inserting this inequality into
$$T_j^\nu u^{\nu}_{j,4}(x)= \int_{^{c}\Bar{R}^{\nu}_{j}}T_j^\nu(x, y)u(y)dy,$$
one can get that for any $x\in Q,$
\begin{eqnarray*}
|T^{\nu}_{j}u^{\nu}_{j,2}(x)|
&\leq& \|u\|_{L^{\infty}}
\int_{^{c}\Bar{R}^{\nu}_{j}}\frac{|\nabla_\xi\varphi(x,\xi^{\nu}_{0})-y|^{N}}{|\nabla_\xi\varphi(x_{0},\xi^{\nu}_{0})-y|^{N}}|T^{\nu}_{j}(x,y)| dy.
\end{eqnarray*}
H\"{o}lder's inequality, Integrating by parts  and Parseval's identity give that
\begin{eqnarray}\label{E5}
|T^{\nu}_{j}u^{\nu}_{j,2}(x)|
&\leq& \|u\|_{L^{\infty}}\big(\int_{^{c}\Bar{R}^{\nu}_{j}}\frac{1}{|\nabla_\xi\varphi(x_{0},\xi^{\nu}_{0})-y|^{2N}} dy\big)^{\frac{1}{2}}
\big(\int_{\mathbb{R}^{n}}\big|\widehat{\partial^{N}_{\xi}a_j^\nu(x, \cdot)}(y
)\big|^{2} dy\big)^{\frac{1}{2}}\nonumber\\
&\lesssim& \|u\|_{L^{\infty}}\big(l2^{j\theta}\big)^{\frac{n}{2}-N}
2^{-j\frac{n}{2}}
|A_{j}\cap\Gamma^{\nu}_{0}|^{\frac{1}{2}}
\lesssim2^{-j\theta(N-\frac{n}{2})}l^{-(N-\frac{n}{2})}\|u\|_{L^{\infty}},
\end{eqnarray}
where the fact that $|A_{j}\cap\Gamma^{\nu}_{0}|\lesssim 2^{jn}$ and (\ref{Ec}) are used. (\ref{E4}) and (\ref{E5}) imply that $T^{\nu}_{j}$ meets the inequality (\ref{E3}). So the proof is finished.

\end{proof}
\section{The proof of main result}
Based on the lemmas above and main idea in \cite{W}, one can finish the proof easily.

%
%

\begin{proof}[Proof of Theorem~\ref{T1}]
Without loss of generality we assume that the symbol $a(x,\xi)$ vanishes for $|\xi|\leq 2$. For the low frequency portion of FIOs, refer to \cite [Theorem 1.18]{Wolfgang}. Let $Q=Q(x_{0},l)$ be a cube centered at $x_{0}$ with side length $l>0$. We aim to prove the estimate
\begin{equation}\label{E0}
\frac{1}{|Q|}\int_{Q}|T_{a,\varphi}u(x)-C_{Q}|dx\lesssim \|u\|_{L^{\infty}},
\end{equation}
where $C_{Q}=\frac{1}{|Q|}\int_{Q}T_{a,\varphi}u(y)dy$. To this end, we decompose the operators $T_{a,\varphi}$ as in (\ref{E1}). If $l>1$, then by Lemma (\ref{L5}), the left-hand side of the above expression can be controlled by
\begin{eqnarray*}
\sum\limits_{j}\frac{2}{|Q(x_{0},l)|}\int_{Q(x_{0},l)}|T_{j}u(x)|dx
&\lesssim&\|u\|_{L^{\infty}}\sum\limits_{j}(2^{-j\frac{n}{2}(1-\delta-\theta)}+2^{-j\theta(N-\frac{n}{2})}l^{-(N-\frac{n}{2})})\\
&\lesssim&\|u\|_{L^{\infty}}.
\end{eqnarray*}

It  therefore remains to handle the case where $l\leq1$. In this case, we can control the left-hand side of (\ref{E0}) by
$$
\sum\limits_{1<2^{j}}\frac{1}{|Q|^{2}}\int_{Q}\int_{Q}|T_{j}u(x)-T_{j}u(z)|dzdx.
$$
We split the sum as
\begin{eqnarray}\label{E13}
\sum\limits_{2^{j}<l^{-1}}
+\sum\limits_{l^{-1}<2^{j}}.
\end{eqnarray}
We analyze each part.
For the first sum, using the fact that
$$|T_{j}u(x)-T_{j}u(z)|\leq\int_{\mathbb{R}^{n}}|T_{j}(x,y)-T_{j}(z,y)|dy$$
and Lemma \ref{L1}, we get that this sum is bounded by
$$\sum\limits_{2^{j}<l^{-1}}2^{j}l\|u\|_{L^{\infty}}\lesssim\|u\|_{L^{\infty}}.$$
For the second sum, we apply the following estimate:
\begin{equation*}
\begin{array}{c}
  \displaystyle\sum\limits_{l^{-1}<2^{j}}\frac{1}{|Q|^{2}}\int_{Q}\int_{Q}|T_{j}u(x)-T_{j}u(z)|dzdx\leq \sum\limits_{l^{-1}<2^{j}}\frac{2}{|Q|^{2}}\int_{Q}|T_{j}u(x)|dx.
\end{array}
\end{equation*}
If $\delta=0$, then Lemma \ref{L2} gives the bound
$$\sum\limits_{l^{-1}<2^{j}}2^{-j\frac{n}{2}(1-\frac{n}{2N})}l^{\frac{n}{2}(\frac{n}{2N}-1)}\|u\|_{L^{\infty}}
\lesssim \|u\|_{L^{\infty}}.$$
If $0<\delta<1$, we further break up this sum as
\begin{equation}\label{E2}
(\sum\limits_{l^{-1}\leq2^{j}\leq l^{-\frac{1}{1-\delta}}}
+\sum\limits_{l^{-\frac{1}{1-\delta}}<2^{j}})\frac{2}{|Q|^{2}}\int_{Q}|T_{j}u(x)|dx.
\end{equation}
Using  Lemma \ref{L4}, the second term can be controlled by $\|u\|_{L^{\infty}}$. For the first term, we write
\begin{eqnarray*}
\sum\limits_{l^{-1}<2^{j}\leq l^{-\frac{1}{1-\delta}}}\frac{2}{|Q|}\int_{Q}|T_{j}u(x)|dx
&=&\big(\sum\limits_{l^{-1}<2^{j}\leq l^{-\frac{1}{\delta}}}
+\sum\limits_{l^{-\frac{1}{\delta}}<2^{j}\leq l^{-\frac{1}{\delta^{2}}}}+...
+\sum\limits_{l^{-\frac{1}{\delta^{k-1}}}<2^{j}\leq l^{-\frac{1}{\delta^{k}}}}\\
&+&...
+\sum\limits_{l^{-\frac{1}{\delta^{\gamma-1}}}<2^{j}\leq \min\{l^{-\frac{1}{1-\delta}},l^{-\frac{1}{\delta^{\gamma}}}\}}\big)
\frac{1}{|Q|}\int_{Q(x_{0},l)}|T_{j}u(x)|dx,
\end{eqnarray*}
where $\gamma$ is the first positive integer such that $\frac{1}{\delta^{\gamma}}\geq \frac{1}{\varrho}$.
Taking $\lambda=\frac{1}{\delta^{k}}$, $k=0,1,...,\gamma-1$ in Lemma \ref{L3}, we see that each sum above is bounded by $\|u\|_{L^{\infty}}$. Thus, we conclude that
\begin{eqnarray*}
\sum\limits_{l^{-1}<2^{j}\leq l^{-\frac{1}{1-\delta}}}\frac{2}{|Q|}\int_{Q}|T_{j}u(x)|dx
\leq C_{\gamma}\|u\|_{L^{\infty}}.
\end{eqnarray*}
Therefore, the proof is completed.
\end{proof}

\bibliographystyle{Plain}

\end{document}